\documentclass[a4paper,11pt,leqno]{article}

\usepackage[latin1]{inputenc}
\usepackage[T1]{fontenc} 
\usepackage[english]{babel}
\usepackage{verbatim}
\usepackage{graphics}
\usepackage{amsfonts}
\usepackage{amsmath}
\usepackage{amsthm}
\usepackage{amssymb}
\usepackage{color}
\usepackage{graphicx}
\usepackage{bm}

\theoremstyle{definition}
\newtheorem{defin}{Definition}[section]

\newtheorem{rem}[defin]{Remark}

\theoremstyle{plane}
\newtheorem{thm}[defin]{Theorem}
\newtheorem{prop}[defin]{Proposition}

\newtheorem{lemma}[defin]{Lemma}

\newcommand{\mbb}{\mathbb}

\newcommand{\mc}{\mathcal}
\newcommand{\veps}{\varepsilon}
\newcommand{\what}{\widehat}
\newcommand{\wtilde}{\widetilde}
\newcommand{\vphi}{\varphi}
\newcommand{\oline}{\overline}

\newcommand{\ra}{\rightarrow}

\newcommand{\vrho}{\varrho}

\newcommand{\R}{\mathbb{R}}
\newcommand{\C}{\mathbb{C}}
\newcommand{\N}{\mathbb{N}}

\newcommand{\z}{\zeta}

\renewcommand{\Re}{{\rm Re}\,}

\newcommand{\Id}{{\rm Id}\,}

\allowdisplaybreaks

\def\d{\partial}

\title{\large{\bfseries{\textsc{Some local questions for hyperbolic systems \\
with non-regular time dependent coefficients}}}}

\author{\normalsize\textsl{Francesco Fanelli} \vspace{.3cm} \\
{\small \textit{Institut Camille Jordan - UMR CNRS 5208}} \vspace{.1cm} \\
{\small \textsc{Universit\'e Claude Bernard -- Lyon 1}} \vspace{.1cm} \\
{\scriptsize {B\^atiment Braconnier}} \\
{\scriptsize {48, Boulevard du 11 novembre 1918}} \\
{\scriptsize {F-69622 Villeurbanne cedex -- FRANCE}} \vspace{.2cm} \\
{\small \ttfamily{fanelli@math.univ-lyon1.fr}} }
\vspace{.2cm}
\date\today

\begin{document}

\maketitle

\subsubsection*{Abstract}
{\small In this note we investigate local properties for microlocally symmetrizable hyperbolic systems with just time dependent coefficients.
Thanks to Paley-Wiener theorem, we establish finite propagation speed by showing precise estimates on the evolution of the support of the solution in terms of suitable norms of the
coefficients of the operator and of the symmetrizer. From this result, local existence and uniqueness follow by quite standard methods.  

Our argument relies on the use of Fourier transform, and it cannot be extended to operators whose coefficients depend also on the space variables.
On the other hand, it works under very mild regularity assumptions on the coefficients of the operator and of the symmetrizer.}

\paragraph*{2010 Mathematics Subject Classification:}{\small 35L40, 
(primary); 35F35, 
35B65, 
35R05 
(secondary).}

\paragraph*{Keywords:}{\small hyperbolic systems; microlocal symmetrizability; finite speed of propagation; local existence and uniqueness;
energy estimates.}

\section{Introduction}

In the present paper, we deal with linear first order hyperbolic systems with time dependent coefficients:
\begin{equation} \label{intro_def:L}
Lu(t,x)\,=\,\d_tu(t,x)\,+\,\sum_{j=1}^nA_j(t)\,\d_ju(t,x)\,,
\end{equation}
where the time variable $t\in[0,T]$, for some $T>0$, and the space variable $x\in\Omega\subset\R^n$. Our main concern is to establish
finite propagation speed and other local properties for operator $L$, under minimal assumptions on the regularity of its coefficients and symmetrizer.

Several physical models can be described, from the mathematical viewpoint, by hyperbolic problems with variable, and especially non-smooth,
coefficients. It is well-known that this lack of regularity may affect the evolution of the system, producing pathological phenomena already at the linear level.

This is the case e.g. for scalar wave equations
\begin{equation} \label{intro_eq:wave}
 Wu(t,x)\,:=\,\d_t^2u(t,x)\,-\,\sum_{j,k=1}^n\d_j\Bigl(a_{jk}(t,x)\,\d_ku(t,x)\Bigr)\,,
\end{equation}
which were extensively studied so far. Among various papers, here we quote \cite{C-DG-S},
\cite{C-L}, \cite{C-M}, \cite{Tar}. All works highlighted that, under very low regularity assumptions on the coefficients of the operator,
the solution loses smoothness in the time evolution, and it becomes more and more irregular when the time goes by.
In particular, this phenomenon has drastic consequences in the ($L^2$ or $\mc{C}^\infty$) well-posedness of the Cauchy problem, which can be recovered only admitting a
loss of a finite number of derivatives. In addition, it makes impossible to capture observability and controllability properties of the operator in a classical sense,
and these notions have to be revisited in order to take into account the loss of regularity.
Finally, counterxemples show that, for too irregular coefficients, the loss of derivatives can be infinite, precluding any possibility of recovering well-posedness or observability,
even in a weak sense.

We refer to \cite{C-DS-F-M_tl}--\cite{C-DS-F-M_wp} (concerning the well-posedness theory) and \cite{F-Z}--\cite{Waters} (concerning observability and control of waves) for an
overview about this topic and for further references, as well as for some recent progress in these directions.

Moreover, we recall that similar phenomena were put in evidence also for transport equations; the literature on the subject is vast, and still evoloving. We refer
e.g. to Chapter 3 of \cite{B-C-D} and the references quoted therein for more informations about this topic.

\medbreak
As for first order hyperbolic systems
\begin{equation} \label{intro_def:P}
Pu(t,x)\,=\,\d_tu(t,x)\,+\,\sum_{j=1}^nA_j(t,x)\,\d_ju(t,x)\,,
\end{equation}
the analysis from the point of view of the minimal smoothness hypotheses is much more recent.
On the other hand, as it is well-known, this context presents some important differences with respect to the case of scalar wave equations.

The most relevant one is that one needs an extra hypothesis about \emph{microlocal symmetrizability} of the system.
Roughly speaking, this means that there exists a scalar product (say) $S$, depending both on $(t,x)$ and on the dual variable $\xi\neq0$,
with respect to which the operator becomes self-adjoint, and then ``classical'' energy estimates work.
For example, strictly hyperbolic systems, or more in general hyperbolic systems with constant multiplicities, are
microlocally symmetrizable (at least for regular coefficients, see also the discussion below), because they are smoothly diagonalizable.

The dependence of $S$ on $\xi$ is usually assumed to be very smooth (i.e. $\mc C^\infty$), while various regularity hypotheses can be considered with respect
to the time and space variables: accordingly, one speaks about bounded (or continuous, or Lipschitz, and so on) symmetrizer.

By a result of Ivri\u{\i} and Petkov (see \cite{Iv-Pet}), the existence of a bounded microlocal symmetrizer is a necessary condition for the Cauchy problem for $P$ to be
well-posed in $\mc{C}^\infty$. As for sufficient conditions, the question is still widely open; the main feature to point out is that
regularity hypotheses are needed \emph{both} on the coefficients of the operator in \eqref{intro_def:P} and on the symmetrizer $S$.

In \cite{M-2008} (see Chapter 7), M\'etivier proved $L^2$ well-posedness for $P$, if the $A_j$'s and the symmetrizer are supposed to be Lipschitz continuous
both in time and space. In particular, such a result extends Friedrichs theory for symmetric systems (see e.g. \cite{Fried_1954}--\cite{Fried_1958}),
which can be recovered in the special instance of $S$ independent of the dual variable $\xi$.
On the other hand, the same author showed in \cite{M_2014} that Lipschitz regularity of $S$ in all its variables (then also in $\xi$) is enough for $L^2$ well-posedness to hold,
if the coefficients of the operator enjoy additional smoothness (namely, $W^{2,\infty}$) in $(t,x)$.

Below the Lipschitz threshold, things go worse and we have to remark that a loss of derivatives produces during the time evolution of the system,
in a similar fashion to what is known for the wave operator \eqref{intro_eq:wave}. This allows one to
recover just $\mc{C}^\infty$ well-posedness, with a finite loss of derivatives: this is the case of work \cite{C-DS-F-M_LL-sys}, where the loss is proved
to be linearly increasing in time, for coefficients and symmetrizer which are log-Lipschitz continuous both in $t$ and $x$. Moreover, explicit counterexamples
establish the sharpness (if one measures regularity just by the modulus of continuity) of the Lipschitz and log-Lipschitz assumptions both on the
$A_j$'s and $S$, for the $L^2$ and the $\mc C^\infty$ well-posedness respectively. These counterexamples are constructed in \cite{M_2014} and in \cite{C-M_2015};
in this last paper, just time-dependent coefficients are considered, like in the case of operator $L$ defined in \eqref{intro_def:L} at the beginning of the introduction.

Still for the case of operator $L$, in \cite{C-DS-F-M_Z-sys} we considered Zygmund and log-Zygmund type assumptions.  For reasons which will appear clear in a while,
we dismissed the hypothesis of microlocal symmetrizability of the system, and we supposed $L$ to be hyperbolic with constant multiplicities.

Zygmund hypotheses are somehow second order conditions, made on the symmetric difference of the function rather than on its modulus of continuity, and they are weaker
than the corresponding ones formulated on the first difference. The special issue is that they are still suitable for well-posedness of hyperbolic Cauchy problems, with the same kind
of loss which would pertain to the latter conditions: so, for pure Zygmund hypothesis one recovers well-posedness in any $H^s$ with no loss of derivatives, while log-Zygmund hypothesis
entails a finite loss, which depends on time.
This picture is very similar to what was already obtained in the case of scalar wave equations \eqref{intro_eq:wave} by Tarama, see paper \cite{Tar} (see also \cite{C-DS-F-M_tl}
and \cite{C-DS-F-M_wp} and the references therein). There, the main point was to compensate the worse behaviour of the coefficients by introducing a lower order corrector
in the definition of the energy, in order to produce special algebraic cancellations in the energy estimates.

For first order systems, the strategy adopted in \cite{C-DS-F-M_Z-sys} was the same: indeed, knowing \textsl{a priori} the existence of a microlocal symmetrizer is out of use in this context,
and the main challange was to build up a suitable microlocal symmetrizer, in a slightly weaker sense. In particular, the construction is just local in $(t,\xi)$, and one has
to forget about global regularity in $\xi$. Such a fact makes some points unclear, or at least less straightforward, in the analysis of
first order hyperbolic systems, when assumptions of Zygmund type are formulated. We will come back to this issue later on.

\medbreak
Among all the above mentioned works, just a few devoted attention to finite propagation speed and other local questions.
In order to establish these results, a classical approach is based on convexifications arguments, combined with a global $L^2$ well-posedness result.

This strategy was adopted in \cite{M_2014} for operator \eqref{intro_def:P} with Lipschitz coefficients and symmetrizer. A crucial point, there, was to show that
the assumed hypotheses are invariant by change of variables, and in particular of the timelike direction.
Then, finite speed of propagation and local uniqueness follow; in addition, the previous property allows for a sharp description of propagation of the supports,
like in \cite{J-M-R_2005} and \cite{Rauch_2005}. Remark that the results in these last papers still require some regularity assumptions on the coefficients, besides
asking \textsl{a priori} a local uniqueness property at any spacelike hypersurface.

The same scheme is followed also in \cite{C-DS-F-M_LL-sys} for operator $P$ with log-Lipschitz coefficients, which admits a symmetrizer which is log-Lipschitz in $(t,x)$, smooth
in $\xi\neq0$. Regularity in $\xi$ is a key ingredient, since the well-posedness issue (with a finite loss of derivatives) exploits paradifferential calculus
in a fundamental way. Moreover, there many efforts are devoted to giving sense to the local Cauchy problem, due to the loss of derivatives
in the energy estimates (a difficulty which was already encountered in \cite{C-M} for wave equations).

Nonetheless, one has to remark that, due to the hypotheses on the regularity of coefficients and symmetrizers, the previous results do not apply to
all the possible cases: in particular, the instance of operator \eqref{intro_def:L} with Zygmund and log-Zygmund coefficients is not covered.
Moreover, it is important to point out that in \cite{C-DG-S}, for the scalar wave operator $W$ with just time dependent coefficients,
local questions were investigated supposing very weak regularity assumptions, i.e. essentially $a_{jk}\in L^1\bigl([0,T]\bigr)$ only.

Motivated by the previous considerations, in the present paper we focus our attention on the case of operator $L$ defined in \eqref{intro_def:L}: we want to address local questions
under the angle of minimal regularity assumptions for the coefficients of the operator and its symmetrizer.
Referring to previous works, like e.g. the above mentioned \cite{M_2014}--\cite{C-DS-F-M_LL-sys}, we have to say that our analysis applies to
operators with just time dependent coefficients, and hence, in particular, we also lose the invariance by change of coordinates.
On the other hand, our main result, on finite propagation speed, have the avantage of asking for very low regularity hypotheses, and to go below the Lipschitz
(or log-Lipschitz) threshold.

Our method of investigation is strongly insipired by the analysis of \cite{C-DG-S} for the scalar operator \eqref{intro_eq:wave}. It is based on fine estimates on the growth of the
solution in the Fourier variable and on the application of Paley-Wiener theorem. We claim no special originality about the techniques of the proof.
As said above, the interest of our study is the very mild smoothness which it requires: basically,
either continuity in time of the $A_j$'s, or global continuity of $S$ on $[0,T]\times\mbb{S}^{n-1}_\xi$, where $\mbb{S}^{n-1}_\xi$ denotes
the unitary sphere in the $\xi$-variable. We refer to Subsection \ref{ss:def-hyp} and Theorem \ref{th:speed} for the precise formulation of
our hypotheses.

The point is that we will be content to estimate the evolution of the support of solutions in a larger functional framework, namely in the space of analytic functionals.
Thanks to this, finite speed of propagation is proved for a wide class of data (see Theorem \ref{th:speed}). After that, we will establish a local theory in the spaces of analytic
functionals and functions (see Section \ref{s:local}). In order to recover corresponding results in finer functional settings, more regularity has to be imposed, of course:
we will not detail this part here, since it would present no special novelties compared to the analysis of \cite{C-DG-S}.

We remark that our hypotheses are still stronger than the ones in \cite{C-DG-S} for wave equations, but, as already pointed out, the case of first order
systems and of wave equations are very different from each other. On the other hand, in analogy with \cite{C-DG-S}, we will see that it is still the $L^1\bigl([0,T]\bigr)$-norm
of the coefficients which enters into play in the propagation of the support of the initial data.

Finally, we point out that the analysis presented here does not apply straight away to operators $L$ with coefficients in Zygmund classes. As a matter of fact, it is true that this
regularity in time is much more than the one demanded in the present paper. But in the Zygmund case one misses a unique global symmetrizer, for which, then, smoothness in $\xi$
fails to hold; in fact, we do not even have a true symmetrizer for our system, but just a family of approximate symmetrizers.
For this reason, the adaptation of the present analysis to operators \eqref{intro_def:L} with Zygmund coefficients would require some other modifications, which
go beyond the scopes of the present paper: we will devote to it a different study. 

\medbreak
Before going on, we give here a brief overview of the paper. 

In Section \ref{s:prelim} we present our general working setting and hypotheses; we also state some preliminaries and well-known results,
which we will exploit in a fundamental way in our study.
Section \ref{s:speed} deals with finite propagation speeed property in the class of analytic functionals; namely, we will state and
prove the main result of the paper (see Theorem \ref{th:speed}). Thanks to this outcome, in Section \ref{s:local} we will develop
a local theory in the classes of analytic functionals and functions.

\subsubsection*{Acknowledgements}

The author whishes to express all his gratitude to T. Alazard, whose relevant question motivated the study presented in this note.

The author is member of the Gruppo Nazionale per l'Analisi Matematica, la Probabilit\`a
e le loro Applicazioni (GNAMPA) of the Istituto Nazionale di Alta Matematica (INdAM).

\subsubsection*{Notations}
We introduce here some notations which will be freely used in the text.

Let the field $\mbb{F}$ be $\R$ or $\C$, and fix a number $n\geq1$.
Let $u$ be a function defined on $\mbb{F}^n$: the variable, if real, will be usually called $x$, while $z$ if complex. Correspondingly, we denote by $\xi\in\R^n$
or $\z\in\C^n$ the respective dual variables. We will also use the decompositions $z=x+iy$ and $\z=\xi+i\eta$, with both $(x,\xi)$ and $(y,\eta)$ in $\R^{2n}$.

Given two vectors $v$ and $w$ in $\C^m$, we will denote by $v\cdot w$ the usual scalar product in $\C^m$ and
by $|v|$ the usual norm of a vector in $\C^m$: namely,
$$
v\,\cdot\,w\,=\,\sum_{j=1}^m v_j\,\oline{w_j}\qquad\mbox{ and }\qquad
|v|^2\,=\,v\,\cdot\,v\,.
$$
On the contrary, given an infinite-dimensional Banach space $X$, we denote by $\|\cdot\|_{X}$ its norm and,
if $X$ is Hilbert, by $(\,\cdot\,,\,\cdot\,)_{X}$ its scalar product. 

The symbol $\mc{M}_m(\mbb{F})$ refers to the set of all $m\times m$ matrices whose components belong to $\mbb{F}$, equipped with the norm
$|\,\cdot\,|_{\mc{M}}$ defined by
$$
|A|_{\mc{M}}\,:=\,\sup_{|v|=1}|Av|\,\equiv\,\sup_{v\neq0}\frac{|Av|}{|v|}\,.
$$
If $A$ is self-adjoint (more in general, if it is a normal matrix), we also have
$$
|A|_{\mc{M}}\,\equiv\,\sup_{|v|=1}|Av\cdot v|
$$

Finally, given two self-adjoint matrices $A$ and $B$ belonging to $\mc{M}_m(\mbb{C})$, we say that $A\leq B$ if the inequality $Av\cdot v\,\leq\,Bv\cdot v$
holds true for all $v\in\C^m$.

\section{Preliminaries} \label{s:prelim}
In this section, we present our work setting and we recall some basic results which we will need in the course of our study. 

\subsection{Basic definitions and hypotheses} \label{ss:def-hyp}
On an infinite strip $[0,T]\times\R^n$, for some time $T>0$ and $n\geq1$, we consider
the $m\times m$ (with $m\geq1$) linear first order system
\begin{equation} \label{def:Lu}
Lu(t,x)\,=\,\d_tu(t,x)\,+\,\sum_{j=1}^nA_j(t)\,\d_ju(t,x)\,.
\end{equation}
Namely, we suppose $u(t,x)\in\R^m$ and, for all $1\leq j\leq n$, the matrices $A_j(t)\in\mc{M}_m(\R)$.

We immediately introduce the symbol $A$ associated to the operator $L$: for all $(t,\xi)\in[0,T]\times\R^n$,
\begin{equation} \label{def:symbol}
A(t,\xi)\,:=\,\sum_{j=1}^n\xi^j\,A_j(t)\,.
\end{equation}
Then, for all $(t,\xi)$, the matrix $A(t,\xi)$ belongs to $\mc{M}_m(\R)$.
Let $\bigl(\lambda_j(t,\xi)\bigr)_{1\leq j\leq m}\subset\C$ denote the family of its eigenvalues at any point $(t,\xi)$.

We recall the following definitions (see e.g. \cite{M-2008}, Chapter 2).
\begin{defin} \label{d:systems}
\begin{itemize}
\item[(i)] ~We say that system \eqref{def:Lu} is \emph{hyperbolic} if, for all $t\in[0,T]$ and all $\xi\neq0$,
the eigenvalues of $A(t,\xi)$ are all real: $\bigl(\lambda_j\bigr)_{1\leq j\leq m}\,\subset\,\R$.

\item[(ii)] ~System \eqref{def:Lu} is \emph{hyperbolic with constant multiplicities} if, for all $t\in[0,T]$ and all
$\xi\neq0$, the eigenvalues of $A(t,\xi)$ are real and semi-simple, with constant multiplicities.
\end{itemize}
\end{defin}

We recall that a (possibly complex) eigenvalue is called \emph{semi-simple} if its algebraic and geometric multiplicities coincide;
a matrix is semi-simple if it is diagonalizable in the complex sense. 

A particular case of hyperbolicity with constant multiplicities is when the operator is \emph{strictly hyperbolic}, i.e. when all the 
eigenvalues are real and distinct (constant multiplicites equal to $1$).

\medbreak
In what follows, we are going to consider a more general assumption than hyperbolicity with constant multiplicities.
More precisely, we require that the system is \emph{uniformly microlocally symmetrizable} in the sense of M\'etivier (see \cite{M-2008},
Chapter 7). The word \textit{uniformly} here refers to $(t,x)\in[0,T]\times\R^n$ (see also Section 4 of \cite{M_2014}).
\begin{defin} \label{d:micro_symm}
Operator $L$, defined in \eqref{def:Lu}, is a \emph{uniformly microlocally symmetrizable hyperbolic system} if there exists a $m\times m$ matrix $S(t,\xi)$,
 homogeneous of degree $0$ in $\xi$, such that:
 \begin{itemize}
  \item for almost every $t$, the map $\xi\,\mapsto\,S(t,\xi)$ is $\mc{C}^\infty$ for $\xi\neq0$;
  \item for any point $(t,\xi)$, the matrix $S(t,\xi)$ is self-adjoint;
  \item there exist constants $0<\lambda\leq\Lambda$ such that $\lambda\,\Id\,\leq\,S(t,\xi)\,\leq\,\Lambda\,\Id$ for any $(t,\xi)$;
  \item for any point $(t,\xi)$, the matrix $S(t,\xi)\,A(t,\xi)$ is self-adjoint.
 \end{itemize}
The matrix valued function $S(t,\xi)$ is called a \emph{(bounded) microlocal symmetrizer} for system \eqref{def:Lu}.
\end{defin}

Remark that, by definition, $S(t,\nu)$ is a smooth, and in particular continuous, function of $\nu\in\mbb{S}^{n-1}$.
No additional regularity is demanded on the time variable at this level: one just requires an $L^\infty$ bound.

Now, we turn our attention to the coefficients of $L$: we suppose that, for all $1\leq j\leq n$,
the matrix-valued functions $A_j$ are $L^1$ in the time variable. More precisely, we require that
\begin{equation} \label{hyp:coeff-L1}
\sum_{j=1}^n\bigl\|A_j\bigr\|_{L^1([0,T];\mc{M}_m(\R))}\,=\,\sum_{j=1}^n\int_0^T\bigl|A_j(t)\bigr|_{\mc{M}}\,dt\,<\,+\infty\,.
\end{equation}
Of course, since we are in finite dimension, this is equivalent to ask a similar property on each component of the matrices $A_j$.

We conclude this part by pointing out a relevant fact about microlocal symmetrizability. It will be useful especially in Section \ref{s:local}
below. It is just an easy remark, but we were not able to find any reference about this property: it seems important to us to 
clearly state it.
\begin{lemma} \label{l:symm}
Let the operator $L$ be defined as in \eqref{def:Lu}. Then $L$ is microlocally symmetrizable, with bounded
symmetrizer $S$, if and only if the adjoint operator
\begin{equation} \label{def:L*}
L^*v(t,x)\,=\,-\,\d_tv(t,x)\,-\,\sum_{j=1}^nA^*_j(t)\,\d_jv(t,x)
\end{equation}
is microlocally symmetrizable, with symmetrizer $S^{-1}$.
\end{lemma}

\begin{proof}
We start by noticing that, up to a minus sign, the symbol of $L^*$ is the matrix
$$
A^*(t,\xi)\,:=\,\sum_{j=1}^n\xi^j\,A^*_j(t)\,.
$$

On the other hand, it is straightforward to check that also $S^{-1}$ fulfills all the properties of Definition \ref{d:micro_symm}
(with constants $\Lambda^{-1}\leq\lambda^{-1}$), up to the last one.
Hence, we have only to verify that $S^{-1}\,A^*$ is self-adjoint.

Now, we remark that we can express the condition $(S\,A)^*\,=\,S\,A$ in the equivalent form $A^*\,=\,S\,A\,S^{-1}$. But one easily has
$$
A^*\,=\,S\,A\,S^{-1}\qquad\Longleftrightarrow\qquad A\,=\,S^{-1}\,A^*\,S\,,
$$
which again is the same as writing $S^{-1}\,A^*\,=\,A\,S^{-1}$. Observing that $A\,S^{-1}\,=\,(S^{-1}\,A^*)^*$
completes the proof of the lemma.
\end{proof}

\begin{rem} \label{r:L*}
The previous result holds true even when the coefficients of $L$ depend both on time and space variables, and the adjoint operator becomes
$$
L^*v(t,x)\,=\,-\,\d_tv(t,x)\,-\,\sum_{j=1}^n\d_j\bigl(A^*_j(t,x)\,\d_jv(t,x)\bigr)\,.
$$
\end{rem}

\subsection{Tools}

We collect here some fundamental results, which will reveal to be fundamental in the sequel.

First of all, let us introduce additional notations for some relevant functional spaces.
Let $\Omega$ be an open subset of $\R^n$; then:
\begin{itemize}
 \item $\mc{H}$ denotes the space of entire functions on $\R^n$;
 \item $\mc{A}(\Omega)$ denotes the space of real analytic functions on $\Omega$;
 \item $\mc{H}'$ is the space of holomorphic functionals over $\C^n$;
 \item $\mc{A}'(\Omega)$ is the space of real analytic functionals over $\Omega$.
\end{itemize}
When $\Omega=\R^n$, we will omit it from the notation. In addition, we will indicate by ${\rm supp}\,u$ the support of a functional $u\in\mc{A}'(\Omega)$.

We recall that, given a compact set $K\subset\Omega$, we say that ${\rm supp}\,u\subset K$ if, for all open neighborhood $U\subset\Omega$ of $K$
and all sequence $(f^k)_k\subset\mc{A}(\Omega)$ such that $f^k\ra0$ in $\mc A(U)$, then one has $<u,f^k>\,\longrightarrow\,0$.

The symbol $<\cdot,\cdot>_{\mc{H}'\times\mc{H}}$ refers to the duality product for the space $\mc{H}'\times\mc{H}$, and an analogous notation will be
used for $\mc{A}'(\Omega)\times\mc{A}(\Omega)$. The spaces in the subscript will be omitted whenever they are evident from the context.

For a functional $u\in\mc{H}'$, we denote by $\what u$ its Fourier transform: for all $(z,\z)\in\C^{2n}$, one has
$$
\what u(\z)\,:=\,<u\,,\,h_\z>\,,\qquad\qquad\mbox{ with }\qquad h_\z(z):=\exp\bigl(-i\,\z\cdot z\bigr)\,.
$$
For $u\in L^1\bigl([0,T];\mc{H}'\bigr)$, we naturally write $\what u(t,\z)\,=\,<u(t)\,,\,h_\z>$ for almost every $t\in[0,T]$.

We recall that a real analytic functional $u$ is \emph{real-valued} if, for any real analytic function $\vphi$ over $\R^n$, the quantity
$<u\,,\,\vphi>\,\in\R$. Whenever $u$ is compactly supported, this is equivalent to require that $\what u(-\xi)=\oline{\what u(\xi)}$ for all
$\xi\in\R^n$.

Fixed some $r>0$, we also set $B(r)$ to be the ball of $\R^n$ centered in the origin and of radius $r$.

\medbreak
This having been pointed out, let us recall the Paley-Wiener Theorem: the present form is the one stated in \cite{C-DG-S}, Section 1.
\begin{thm} \label{th:PW}
(I) Let $u\in\mc{H}'$ and $r>0$. Then $u\in\mc{A}'$, with ${\rm supp}\,u\,\subset\,B(r)$, if and only if the following condition is verified: for all
$\delta>0$, there exists a constant $C_\delta>0$ such that
$$
\left|\what{u}(\z)\right|\,\leq\,C_\delta\,\exp\bigl(\delta\,|\xi|\,+\,(r+\delta)\,|\eta|\bigr)
$$
for all $\z=\xi+i\eta\,\in\,\C^n$, with $|\z|\geq1$.

(II) Let $u\in L^1\bigl([0,T];\mc{A}'\bigr)$, with support contained in the ball $B(r)$. Then,  for all
$\delta>0$, there exists a constant $C_\delta>0$ such that
$$
\int_0^T\left|\what{u}(t,\z)\right|\,dt\,\leq\,C_\delta\,\exp\bigl(\delta\,|\xi|\,+\,(r+\delta)\,|\eta|\bigr)
$$
for all $\z=\xi+i\eta\,\in\,\C^n$, with $|\z|\geq1$.
\end{thm}

The previous theorem will be the fundamental tool to investigate propagation of the support of the initial data,
and then to derive the finite propagation speed property (see Subection \ref{s:speed} below).

On the other hand, in order to make this argument consistent, we need a general existence result for operator \eqref{def:Lu}.
This is provided by the following statement, which is in the same spirit of the celebrated Cauchy-Kovalevska Theorem.
We refer to Theorem 1 of \cite{C-DG-S} and Theorem 17.2 of \cite{Treves} for analogous results.
\begin{thm} \label{th:CK}
Let $L$ be the operator defined in \eqref{def:Lu}, and suppose that its coefficients $\bigl(A_j\bigr)_j$ belong to the space $L^1\bigl([0,T];\mc{M}_m(\R)\bigr)$
(so that \eqref{hyp:coeff-L1} is verified). 

Then, for any $u_0\in\mc{H}'$ and $f\in L^1\bigl([0,T];\mc{H}'\bigr)$, the Cauchy problem
\begin{equation} \label{eq:Cauchy}
\left\{\begin{array}{l}
        Lu\;=\;f \\[1ex]
	u_{|t=0}\;=\;u_0\,,
       \end{array}\right. 
\end{equation}
admits a unique solution $u\in\mc{C}\bigl([0,T];\mc{H}'\bigr)$.
\end{thm}

The proof is analogous to the one given in Appendix B of \cite{C-DG-S}, and it is based on a Picard iteration scheme. Actually, in that paper the result is
stated for scalar wave operators, but the proof consists in recasting the equation as a first order system. Here, the only difference is the regularity
in time of the solution, which is obtained by an inspection of the equation $Lu=f$.

\begin{rem} \label{r:CK}
We remark that \emph{no hyperbolicity} hypotheses are needed on the system, for the general existence result in the space of holomorphic functionals.
\end{rem}

\section{Finite propagation speed and related results} \label{s:speed}

In the present section we state and prove our main result, concerning \emph{local existence} and \emph{finite propagation speed} for
hyperbolic systems which are microlocally symmetrizable, in the sense specified by Definition \ref{d:micro_symm}.

\begin{thm} \label{th:speed}
 Let $L$ be the operator defined by formula \eqref{def:Lu}, with coefficients $\bigl(A_j\bigr)_j$ 
verifying property \eqref{hyp:coeff-L1}. Suppose that $L$ is microlocally symmetrizable, in the sense of Definition \ref{d:micro_symm},
and let $S$ be its symmetrizer. \\
Suppose also that one of the following conditions holds true:
\begin{itemize}
 \item[(i)] ~ either the map $(t,\nu)\,\mapsto\,S(t,\nu)$ is continuous on $[0,T]\times\mbb{S}^{n-1}$, or
 \item[(ii)] ~ the coefficients $\bigl(A_j\bigr)_j$ are uniformly bounded, i.e. there exists $C_0>0$ such that
\begin{equation} \label{est:bounded}
\sum_{j=1}^n\;\sup_{t\in[0,T]}\left|A_j(t)\right|_{\mc M}\,\leq\,C_0\,.
\end{equation}
\end{itemize}
Finally, assume that $u_0\in\mc{A}'$ and $f\in L^1\bigl([0,T];\mc{A}'\bigr)$, with ${\rm supp}\,u_0$ and ${\rm supp}\,f(t)$ (for almost
every $t$) contained in the ball $B(r_0)$, for some radius $r_0>0$.

Then the solution $u$ to problem \eqref{eq:Cauchy} belongs to $\mc{C}\bigl([0,T];\mc{A}'\bigr)$. Moreover, defined the quantities
$$
\alpha(t)\,:=\,\sum_{j=1}^n\bigl|A_j(t)\bigr|_{\mc{M}}\qquad\mbox{ and }\qquad r(t)\,:=\,r_0\,+\,2\sqrt{\Lambda}\int_0^t\alpha(\tau)\;d\tau\,,
$$
then, for all $t\in[0,T]$, one has that $\;{\rm supp}\,u(t)\,\subset\,B\bigl(r(t)\bigr)$.
 \end{thm}


\begin{rem} \label{r:adj}
Notice that the hypotheses of the previous theorem are invariant under considering the adjoint problem (recall Lemma \ref{l:symm}).
\end{rem}

The rest of the present section is devoted to the proof of Theorem \ref{th:speed}.
Let us recall that the existence of a unique solution $u\in\mc{C}\bigl([0,T];\mc{H}'\bigr)$ to problem \eqref{eq:Cauchy} is
guaranteed by Theorem \ref{th:CK} above. Then, thanks to Theorem \ref{th:PW}, we have just to establish suitable energy estimates for $u$.

The strategy is the following: first of all, we smooth out the symmetrizer with respect to the time variable. Then, we define an approximate energy function in the Phase Space,
for which we show precise estimates. Finally, we conclude by linking the Fourier variable with the approximation parameter and by applying Paley-Wiener Theorem.

Before starting, let us state an important remark.

\begin{rem} \label{r:reg-S}
Notice that, since the coefficients just depend on time, the dual variable is treated as a parameter at this level.
In particular, we do not exploit the $\mc C^\infty$ regularity of the symmetrizer with respect to $\xi$.

More precisely, as for the dependence on the dual variable, we will use the following facts only:
\begin{enumerate}
 \item the continuity for $\nu\in\mbb{S}^{n-1}$, in the next paragraph, in order to have a uniform decay of the function $\omega_S$ (see its definition below);
 \item hypothesis \textit{(i)} of Theorem \ref{th:speed}, in Subsection \ref{ss:end} below, in order to get the final estimate
 without knowing \textsl{a priori} $L^\infty$ bounds on the coefficients. 
\end{enumerate}
\end{rem}

\subsection{Approximation of the coefficients} \label{ss:approx}

First of all, let us extend the function $S(\cdot,\xi)$ (for any fixed $\xi\neq0$) out of $[0,T]$, by taking its mean value on this interval, i.e. the quantity
$$
\mu_S(\xi)\,:=\,\frac{1}{T}\,\int^T_0|S(t,\xi)|_{\mc M}\,dt\,.
$$
For convenience, we keep the notation $S(t,\xi)$ even to denote the symmetrizer extended to all times $t\in\R$.

Furthermore, for $0<\sigma<T$, let us introduce the function
$$
\omega_S(\xi,\sigma)\,:=\,\sup_{\tau\in[0,\sigma]}\int^{T-\sigma}_0|S(t+\tau,\xi)\,-\,S(t,\xi)|_{\mc M}\,dt\,.
$$
Notice that, $S$ being smooth on the unitary sphere $\mbb{S}^{n-1}$ of the Phase Space,
the function $\omega_S(\xi,\sigma)\,\longrightarrow\,0$ for $\sigma\ra0$, uniformly for $|\xi|=1$.

Thanks to Proposition 1 of \cite{C-DG-S} (see Appendix D of that paper), for all $0\leq\sigma\leq T/2$ and all $\xi\neq0$ fixed, the following estimate holds true:
\begin{equation} \label{est:approx-S}
\sup_{\tau\in[0,\sigma]}\int^T_0\left|S(t+\tau,\xi)-S(t,\xi)\right|_{\mc M}\,dt\,\leq\,
2\left(\omega_S(\xi,\sigma)+\frac{\sigma}{T}\int^T_0\left|S(t,\xi)\right|_{\mc M}\,dt\right).
\end{equation}

Now, for reasons which will become clear in a while, we need to smooth out the symmetrizer by convolution with a mollifier kernel.
More precisely, we take an even function $\rho\in\mc{C}^\infty_0(\mbb{R})$, with $0\leq\rho\leq1$, whose support is contained in the interval $[-1,1]$ and
such that $\int\rho(t)\,dt=1$; then we define the family
$$
\rho_\veps(t)\,:=\,\frac{1}{\veps}\,\,\rho\!\left(\frac{t}{\veps}\right)\qquad\qquad\forall\,\veps\in\,]0,1]\,.
$$
Next, for all $\veps\in\,]0,1]$ we set
\begin{equation} \label{eq:S_e}
S_\veps(t,\xi)\,:=\,\bigl(\rho_\veps\,*_t\,S(\cdot,\xi)\bigr)(t)\,=\,\int_{\mbb{R}_\tau}\rho_{\veps}(t-\tau)\,S(\tau,\xi)\,d\tau\,,
\end{equation}
which is a smooth function both of $t\in\R$ and $\xi\in\R^n\setminus\{0\}$ (separately).

The family $\bigl(S_\veps\bigr)_\veps$ enjoys the following properties.
\begin{lemma} \label{l:S_e}
\begin{enumerate}
\item ~For each $\veps\in\,]0,1]$, the matrix $S_\veps$ is still self-adjoint, positive definite and bounded. In particular,
for any $(t,\xi)$ one has
$$\lambda\leq S_\veps(t,\xi)\leq\Lambda\,,$$
for the same positive constants $\lambda$ and $\Lambda$ pertaining to the symmetrizer $S$. 
\item ~One also has the following estimates:
\begin{eqnarray*}
\int^T_0\left|S_\veps(t,\xi)-S(t,\xi)\right|_{\mc M}\,dt & \leq & C\,\left(\omega_S(\xi,\veps)\,+\,\veps\,\sqrt{\Lambda}\right) \\
\int^T_0\left|\d_tS_\veps(t,\xi)\right|_{\mc M}\,dt & \leq & \frac{C}{\veps}\,\left(\omega_S(\xi,\veps)\,+\,\veps\,\sqrt{\Lambda}\right)
\end{eqnarray*}
for a constant $C>0$ just depending on the $L^1$ norms of $\rho$ and $\rho'$.
\end{enumerate}
\end{lemma}

\begin{proof}
Thep properties stated in point $1.$ are straightforward. So, let us focus on the inequalities of point $2.$ in the previous statement: we
start by proving the first one.

From the definition of $S_\veps$, one easily infers
\begin{eqnarray*}
\int^T_0\left|S_\veps(t,\xi)-S(t,\xi)\right|_{\mc M}\,dt & = & \int^T_0\frac{1}{\veps}\left|\int_{|\tau|\leq\veps}
\rho(\tau/\veps)\,\bigl(S(t+\tau,\xi)-S(t,\xi)\bigr)\,d\tau\right|_{\mc M}\,dt \\
& \leq & \|\rho\|_{L^1}\,\sup_{\tau\in[0,\veps]}\int^T_0\left|S(t+\tau,\xi)-S(t,\xi)\right|_{\mc M}\,dt\,.
\end{eqnarray*}
Now, it is just a matter of using estimate \eqref{est:approx-S} and the uniform bound $|S(t,\xi)|_{\mc M}\leq\sqrt{\Lambda}$; then,
the inequality immediately follows.

Concerning the estimate involving $\d_tS_\veps$, we argue in a very similar way. Namely, using the fact that $\int\rho'\,=\,0$,
we can write
$$
\int^T_0\left|\d_tS_\veps(t,\xi)\right|_{\mc M}\,dt\,=\,\int^T_0\frac{1}{\veps^2}\left|\int_{|\tau|\leq\veps}
\rho'(\tau/\veps)\,\bigl(S(t+\tau,\xi)-S(t,\xi)\bigr)\,d\tau\right|_{\mc M}\,dt\,;
$$
the right-hand side of this identity can be buonded as done above, giving us
$$
\int^T_0\left|\d_tS_\veps(t,\xi)\right|_{\mc M}\,dt\,\leq\,\frac{1}{\veps}\,\|\rho'\|_{L^1}\,\sup_{\tau\in[0,\veps]}\int^T_0\left|S(t+\tau,\xi)-S(t,\xi)\right|_{\mc M}\,dt\,.
$$
Now, one concludes thanks to \eqref{est:approx-S} again.
\end{proof}

\begin{rem} \label{r:inverse}
Observe that, from point $(1)$ of the previous lemma, we easily deduce the inequality
$\left|S^{-1}_\veps\right|_{\mc M}\,\leq\,1/\sqrt{\lambda}$.
\end{rem}

\subsection{Energy estimates}
Let $u\in\mc{C}\bigl([0,T];\mc{H}'\bigr)$ be the solution to problem \eqref{eq:Cauchy}, whose existence is guaranteed by 
Theorem \ref{th:CK}. Our goal is to apply Theorem \ref{th:PW}, after establishing suitable energy estimates for $u$.

To this end, we start by proving a preliminary result.

\begin{prop} \label{p:speed-est}
  Let $L$ be the operator defined by formula \eqref{def:Lu}, with coefficients $\bigl(A_j\bigr)_j$ 
verifying property \eqref{hyp:coeff-L1}. Suppose that $L$ is microlocally symmetrizable, in the sense of Definition \ref{d:micro_symm},
and let $S$ be its symmetrizer. Finally, for $\z=\xi+i\eta\,\in\,\C^n$, define the function
$$ 
\vphi_\veps(t,\z)\,:=\,\frac{2}{\sqrt{\lambda}}\,\left|\d_tS_\veps(t,\xi)\right|_{\mc M}\,+\,
\frac{2}{\sqrt{\lambda}}\,\left|S_\veps(t,\xi)-S(t,\xi)\right|_{\mc M}\,\alpha(t)\,|\xi|\,+\,
2\,\sqrt{\Lambda}\,\alpha(t)\,|\eta|\,,
$$ 
where $\alpha$ is the quantity defined in the statement of Theorem \ref{th:speed} and $\bigl(S_\veps\bigr)_\veps$ is the smooth approximation
introduced in \eqref{eq:S_e}.

Then there exists a constant $C>0$ such that, for all $t\in[0,T]$ and all $|\z|\geq1$, one has the estimate
$$
\left|\what{u}(t,\z)\right|\,\leq\,C\,\frac{\sqrt{\Lambda}}{\sqrt{\lambda}}\,\exp\left(\int_0^t\vphi_\veps(\tau,\z)\,d\tau\right)\,
\left(\left|\what{u}_0(\z)\right|\,+\,\int_0^T\left|\what{f}(\tau,\z)\right|\,d\tau\right)\,.
$$
where $u$ is the solution to \eqref{eq:Cauchy}, given by Theorem \ref{th:CK}.
\end{prop}

\begin{rem} \label{r:speed-est}
Notice that the previous proposition holds true even without the additional hypotheses $(i)$-$(ii)$ of Theorem \ref{th:speed}. Indeed, these conditions are
needed just in the final part of the proof, to control $\vphi_\veps$ in the argument of the exponential factor.
\end{rem}

\begin{proof}[Proof of Proposition \ref{p:speed-est}]
Let us introduce an approximated energy for $u$, defined in Fourier variables: for $\z=\xi+i\eta$, we set
\begin{equation} \label{eq:energy}
E_\veps(t,\z)\,:=\,S_\veps(t,\xi)\,\what{u}(t,\z)\,\cdot\,\what{u}(t,\z)\,.
\end{equation}
We remark that it is important to keep $S_\veps$ depending just on $\Re\z=\xi$, not to lose self-adjointness\footnote{For instance, already for $m=1$, take the function
$S(\xi)=\xi/|\xi|$: then $S(\z)$ is well-defined, but it is no more self-adjoint.}
and symmetrizability\footnote{We know that $S(t,\xi)$ is a symmetrizer for $A(t,\xi)$, but this does not imply that $S(t,\z)$ is
a symmetrizer for $A(t,\z)=A(t,\xi)+iA(t,\eta)$.} properties.

Let us point out that, for all $(t,\z)\,\in\,\R\times\C^n$, with $\Re\z=\xi\neq0$, one has
\begin{equation} \label{est:en-equiv}
\lambda\,\left|\what{u}(t,\z)\right|^2\,\leq\,E_\veps(t,\z)\,\leq\,\Lambda\,\left|\what{u}(t,\z)\right|^2\,.
\end{equation}

Next, let us compute the time derivative of the approximated energy $E_\veps$. Keep in mind that, from equation $Lu=f$, if
we pass in the Phase Space we get
\begin{equation} \label{eq:fourier}
\d_t\what{u}(t,\z)\,+\,i\,A(t,\z)\,\what{u}(t,\z)\,=\,\what{f}(t,\z)\,,
\end{equation}
where, by \eqref{def:symbol}, we have denoted $A(t,\z)=A(t,\xi)+iA(t,\eta)$. Therefore, differentiating \eqref{eq:energy}
with respect to time we find
\begin{eqnarray*}
\d_tE_\veps & = & \d_tS_\veps\what{u}\cdot\what{u}\,+\,2\,\Re\left(S_\veps\d_t\what{u}\cdot\what{u}\right) \\
& = & \d_tS_\veps\what{u}\cdot\what{u}\,+\,2\,\Re\left(S_\veps\what{f}\cdot\what{u}\right)\,+\,
2\,\Re\left(-iS_\veps A\what{u}\cdot\what{u}\right)\,.
\end{eqnarray*}
Recall that, in the last term, $S_\veps=S_\veps(t,\xi)$ while $A=A(t,\z)$.

We start estimating each term in the right-hand side of the last equation. First of all, applying Cauchy-Schwarz inequality
to the $S_\veps$ scalar product, we find
$$
\left|\Re\left(S_\veps\what{f}\cdot\what{u}\right)\right|\,\leq\,\left(S_\veps\what{f}\cdot\what{f}\right)^{1/2}\,
\left(S_\veps\what{u}\cdot\what{u}\right)^{1/2}\,
$$
which immediately implies the bound
\begin{equation} \label{est:f}
\left|2\,\Re\left(S_\veps\what{f}\cdot\what{u}\right)\right|\,\leq\,2\,\sqrt{\Lambda}\,\bigl|\what{f}\,\bigr|\,\left(E_\veps\right)^{1/2}\,.
\end{equation}

Let us consider now the term $\d_tS_\veps\what{u}\cdot\what{u}$: writing $\d_tS_\veps=\d_tS_\veps\,(S_\veps)^{-1}\,S_\veps$
and recalling Remark \ref{r:inverse}, we easily get the estimate
\begin{equation} \label{est:d_tS}
\left|\d_tS_\veps\what{u}\cdot\what{u}\right|\,\leq\,\frac{2}{\sqrt{\lambda}}\,\left|\d_tS_\veps\right|_{\mc M}\,E_\veps\,.
\end{equation}

Finally, we deal with the last term which appears in the expression for $\d_tE_\veps$. Decomposing $A(t,\z)=A(t,\xi)+iA(t,\eta)$
and recalling that $S(t,\xi)\,A(t,\xi)$ is self-adjoint, we infer
\begin{eqnarray*}
\Re\left(-iS_\veps(t,\xi)A(t,\z)\what{u}\cdot\what{u}\right) & = & \Re\left(-iS_\veps(t,\xi)A(t,\xi)\what{u}\cdot\what{u}\right)+
\Re\left(S_\veps(t,\xi)A(t,\eta)\what{u}\cdot\what{u}\right) \\
& = & \Re\left(-i\bigl(S_\veps(t,\xi)-S(t,\xi)\bigr)A(t,\xi)\what{u}\cdot\what{u}\right)+
\Re\left(S_\veps(t,\xi)A(t,\eta)\what{u}\cdot\what{u}\right).
\end{eqnarray*}
The former term in the right-hand side of this last equality can be controlled by the quantity
$$
\frac{1}{\sqrt{\lambda}}\,\left|S_\veps(t,\xi)-S(t,\xi)\right|_{\mc M}\,\sum_j\left|A_j(t)\right|_{\mc M}\,|\xi|\,E_\veps\,;
$$
as for the latter, instead, we have
$$
\left|\Re\left(S_\veps(t,\xi)A(t,\eta)\what{u}\cdot\what{u}\right)\right|\,\leq\,\sqrt{\Lambda}\,\sum_j\left|A_j(t)\right|_{\mc M}\,
|\eta|\,E_\veps\,.
$$
Putting these last inequalities together gives us
\begin{eqnarray} \label{est:S_e-S}
\left|2\,\Re\left(-iS_\veps A\what{u}\cdot\what{u}\right)\right| & \leq &
\frac{2}{\sqrt{\lambda}}\,\left|S_\veps(t,\xi)-S(t,\xi)\right|_{\mc M}\,\sum_j\left|A_j(t)\right|_{\mc M}\,|\xi|\,E_\veps\,+   \\
& & \qquad\qquad\qquad\qquad\qquad\qquad\qquad +\,2\,\sqrt{\Lambda}\,\sum_j\left|A_j(t)\right|_{\mc M}\,|\eta|\,E_\veps\,. \nonumber
\end{eqnarray}

Therefore, from \eqref{est:f}, \eqref{est:d_tS} and \eqref{est:S_e-S} we deduce the estimate
\begin{eqnarray*}
\d_tE_\veps & \leq & 2\,\sqrt{\Lambda}\,\bigl|\what{f}\,\bigr|\,\left(E_\veps\right)^{1/2}\,+\,\frac{2}{\sqrt{\lambda}}\,\left|\d_tS_\veps\right|_{\mc M}\,E_\veps\,+ \\
& & \quad +\left(\frac{2}{\sqrt{\lambda}}\,\left|S_\veps(t,\xi)-S(t,\xi)\right|_{\mc M}\,\alpha(t)\,|\xi|\,+\,
2\,\sqrt{\Lambda}\,\alpha(t)\,|\eta|\right)E_\veps\,.
\end{eqnarray*}

Now, if we define $e_\veps(t,\z)\,=\,\bigl(E_\veps(t,\z)\bigr)^{1/2}$, an application of Gronwall Lemma to the previous inequality allows us to find
$$
e_\veps(t,\z)\,\leq\,e^{\int_0^t\vphi_\veps(\tau,\z)\,d\tau}\,\left(e_\veps(0,\z)\,+\,2\,\sqrt{\Lambda}\,\int_0^t\left|\what{f}(\tau,\z)\right|\,d\tau\right)\,;
$$
thanks to \eqref{est:en-equiv}, we obtain the desired estimate for $\Re\z\neq0$.

On the other hand, setting $S(t,0)=0$, the same estimate still holds true when $\Re\z=\xi=0$. Indeed, it is enough to multiply equation \eqref{eq:fourier}
by $\what u$ and use Gronwall inequlity again, recalling that $\Lambda/\lambda\geq1$.
\end{proof}

\subsection{End of the proof} \label{ss:end}
Let us come back to the proof of Theorem \ref{th:speed}.

Our starting point is the inequality established in Proposition \ref{p:speed-est} above. By our hypotheses, combining it with Theorem \ref{th:PW}, we obtain the
following property: for any $\delta>0$ fixed, there exists a constant $C_\delta>0$ such that
$$
\left|\what{u}(t,\z)\right|\,\leq\,C\,C_\delta\,\frac{\sqrt{\Lambda}}{\sqrt{\lambda}}\,\exp\left(\int_0^t\vphi_\veps(\tau,\z)\,d\tau\right)\,
\exp\bigl(\delta\,|\xi|\,+\,(r_0+\delta)\,|\eta|\bigr)
$$
for any $\z=\xi+i\eta\,\in\,\C^n$, with $|\z|\geq1$. So, our next goal is to find suitable bounds for the integral $\int_0^t\vphi_\veps(\tau,\z)\,d\tau$,
or better for the quantities
\begin{eqnarray*}
\mc I_1 & := & \frac{2}{\sqrt{\lambda}}\,\int^T_0\left|\d_tS_\veps(t,\xi)\right|_{\mc M}\,dt \\
\mc I_2 & := & \frac{2\,|\xi|}{\sqrt{\lambda}}\,\int^T_0\left|S_\veps(t,\xi)-S(t,\xi)\right|_{\mc M}\,\alpha(t)\,dt\,.
\end{eqnarray*}

Before proceeding, let us make the fundamental choice
$$
\veps\,=\,\frac{1}{|\z|}\,.
$$
We also remark that, without loss of generality, we can limit ourselves to consider the case $\xi\neq0$, the instance $\xi=0$ being actually simple
(recall that we have set $S(t,\xi)=0$ when $\xi=0$).

First of all, from item $2.$ of Lemma \ref{l:S_e} we immediately deduce the estimate
\begin{equation} \label{est:I_1}
\mc I_1\,\leq\,C\,\bigl(|\z|\,\omega_S(\xi,1/|\z|)\,+\,1\bigr)\,,
\end{equation}
for a new positive constant $C$ also depending on $\lambda$ and $\Lambda$.

As for $\mc I_2$, let us start by assuming hypothesis $(ii)$ of Theorem \ref{th:speed} for a while. By Lemma \ref{l:S_e} again, we get
\begin{equation} \label{est:I_2}
\mc I_2\,\leq\,C\,|\xi|\,\left(\omega_S(\xi,1/|\z|)\,+\,|\z|^{-1}\,\right)\,\leq\,C\,\bigl(|\z|\,\omega_S(\xi,1/|\z|)\,+\,1\bigr)\,,
\end{equation}
where, this time, the constant $C$ depends on $\lambda$, $\Lambda$ and $C_0$.

Thanks to \eqref{est:I_1} and \eqref{est:I_2}, we find, for new suitable constants $C$ and $C_\delta$,
\begin{eqnarray*}
\left|\what{u}(t,\z)\right| & \leq & C_\delta\,\exp\bigl(C\,|\z|\,\omega_S(\xi,1/|\z|)\bigr)\,
\exp\biggl(\delta\,|\xi|\,+\,\bigl(r(t)+\delta\bigr)\,|\eta|\biggr) \\
& \leq & C_\delta\;\mc{E}_\delta(\zeta)\;\exp\biggl(2\delta\,|\xi|\,+\,\bigl(r(t)+2\delta\bigr)\,|\eta|\biggr)\,
\end{eqnarray*}
where, after setting $\wtilde{\omega}_S(1/|\z|)\,:=\,\sup_{|\nu|=1}\omega_S(\nu,1/|\z|)$, we have defined
$$
\mc{E}_\delta(\z)\,:=\,\exp\biggl(C\,\bigl(\wtilde{\omega}_S(1/|\z|)\,-\,\delta\bigr)\,|\z|\biggr)\,.
$$

Now, recalling that $\wtilde{\omega}_S(1/|\z|)\,\longrightarrow\,0$ for $|\z|\,\ra\,+\infty$, one easily deduces that the function
$\mc{E}_\delta(\z)$ is bounded on $\left\{|\z|\geq1\right\}$, for all $\delta>0$ fixed. Hence the previous inequality becomes
$$
\left|\what{u}(t,\z)\right|\,\leq\,C'\,C_\delta\;\exp\biggl(2\delta\,|\xi|\,+\,\bigl(r(t)+2\delta\bigr)\,|\eta|\biggr)\,,
$$
for a new positive constant $C'$.
Keeping in mind Theorem \ref{th:PW}, this estimate concludes the proof of Theorem \ref{th:speed} under hypothesis $(ii)$.

\medbreak
Under hypothesis $(i)$, the only thing which changes is that inequlity \eqref{est:I_2} does not hold true anymore. Then, we need a new control on the $\mc I_2$ term.

First of all, we remark that, being $S$ continuous on the compact set $[0,T]\times\mbb{S}^{n-1}$, it is uniformly continuous.
So, for $\delta>0$ fixed above (given by Paley-Wiener Theorem), there exists  an $\veps_\delta>0$ such that, if $\veps\leq\veps_\delta$, then
$$
\left|S_\veps(t,\xi)-S(t,\xi)\right|_{\mc M}\,=\,
\frac{1}{\veps}\left|\int_{|\tau|\leq\veps}
\rho(\tau/\veps)\,\bigl(S(t+\tau,\xi)-S(t,\xi)\bigr)\,d\tau\right|_{\mc M}\,\leq\,C\,\delta\,.
$$
Remark that $\veps_\delta$ and $C$ are both independent of $\xi$.

Therefore, if $|\z|\geq1/\veps_\delta$ one gathers
$$
\mc I_2\,\leq\,C\,\alpha(T)\,|\xi|\,\delta\,\leq\,C\,\delta\,|\xi|\,,
$$
where again $C$ depends on $\lambda$. On the other hand, whenever $|\z|\leq1/\veps_\delta$, then it is easy to see that
$\mc I_2\,\leq\,C$, for some constant $C=C(\delta)$ depending also on $\delta$ (and on $\Lambda$, $\lambda$).

In the end, putting this inequality together with \eqref{est:I_1}, we arrive at the bound
$$
\left|\what{u}(t,\z)\right|\,\leq\,C_\delta\;\mc{E}_\delta(\zeta)\;\exp\biggl((2+C)\delta\,|\xi|\,+\,\bigl(r(t)+2\delta\bigr)\,|\eta|\biggr)\,,
$$
which allows us to conclude as before.

Theorem \ref{th:speed} is now completely proved.


\section{Local theory in analytic functional classes} \label{s:local}

The main goal of the present paragraph is to derive local uniqueness of solutions to problem \eqref{eq:Cauchy}, where $L$ is the operator
defined in \eqref{def:Lu}.

First of all, we need the counterpart of Theorem \ref{th:speed} for data which are not compactly supported in $\R^n$. We limit ourselves to
state an existence result in the space of real-analytic functions, which is the only required part for our next study.

\begin{thm} \label{th:A}
 Let $L$ be the operator defined by formula \eqref{def:Lu}, with coefficients $\bigl(A_j\bigr)_j$ 
verifying property \eqref{hyp:coeff-L1}. Suppose that $L$ is microlocally symmetrizable, in the sense of Definition \ref{d:micro_symm},
and let $S$ be its symmetrizer. \\
Suppose also that either condition $(i)$ or condition $(ii)$ of Theorem \ref{th:speed} is fulfilled. \\
Finally, assume that $u_0\in\mc{A}$ and $f\in L^1\bigl([0,T];\mc{A}\bigr)$.

Then, there exists a solution $u$ to problem \eqref{eq:Cauchy}, with $u\in\mc{C}\bigl([0,T];\mc A\bigr)$.
\end{thm}

The proof of the previous statement goes along the main lines of the proof to Theorem 4 in \cite{C-DG-S}, and so it is omitted here. It is based on a duality argument
(keep in mind Lemma \ref{l:symm} stated above)
and the application of Theorem \ref{th:speed} proved before.

\begin{rem} \label{r:real}
We point out that, if $u_0$ and $f$ are real valued functionals (or functions), then the corresponing solution $u$, given by Theorem \ref{th:speed}
(respectively, by Theorem \ref{th:A}), is also real valued.

The proof of this fact is straightforward. In the $\mc{A}'$ case, since $u_0$ and $f$ are real valued, then both $\what{u}(t,\xi)$ and $\oline{\what{u}(t,-\xi)}$
are solutions to \eqref{eq:fourier} with $\eta=0$ (i.e. $\zeta=\xi$). Hence, the conclusion follows by the uniqueness part in Theorem \ref{th:CK}.
In the $\mc A$ case, one has to argue by duality, as in \cite{C-DG-S} (see Remark 6 of that paper).
\end{rem}

Thanks to the previous theorem, we can get a result on \emph{continuous dependence} of the solution on the data.
\begin{lemma} \label{l:dependence}
 Let $L$ be the operator defined by formula \eqref{def:Lu}, with coefficients $\bigl(A_j\bigr)_j$ 
verifying property \eqref{hyp:coeff-L1}. Suppose that $L$ is microlocally symmetrizable, in the sense of Definition \ref{d:micro_symm},
and let $S$ be its symmetrizer. \\
Suppose also that either condition $(i)$ or condition $(ii)$ of Theorem \ref{th:speed} is fulfilled. \\
Let $r_0>0$ a real number, and denote
\begin{equation} \label{def:rho}
\vrho(t)\,:=\,r_0\,-\,2\sqrt{\Lambda}\int_0^t\alpha(\tau)\;d\tau\,,
\end{equation}
where the function $\alpha(t)$ is defined in Theorem \ref{th:speed}. Set $B_0=B(r_0)$ and $B_t=B\bigl(\vrho(t)\bigr)$. \\
Finally, take sequences $\bigl(u_0^k\bigr)_k\subset\mc{A}$ and $\bigl(f^k)_k\subset L^1\bigl([0,T];\mc{A}\bigr)$, and let
$\bigl(u^k\bigr)_k\subset\mc{C}\bigl([0,T];\mc A\bigr)$ be the sequence of corresponding solutions, given by Theorem \ref{th:A}.

If $u_0^k\ra0$ in $\mc{A}(B_0)$ and $f^k\ra0$ in $L^1\bigl([0,T];\mc{A}(B_0)\bigr)$, then one has
$$
u^k\,\longrightarrow\,0\qquad\quad\mbox{ in }\qquad\mc{C}\bigl([0,t];\mc{A}(B_t)\bigr)\,,
$$
for all $t\in\,]0,T]$ such that $\vrho(t)>0$.
\end{lemma}

\begin{proof}
Fix $t\in\,]0,T]$ such that $\vrho(t)>0$. Let $g\in L^1\bigl([0,t];\mc{A}'(B_t)\bigr)$ and let $v\in\mc{C}\bigl([0,t];\mc{A}'\bigr)$
the solution of the dual system
$$ 
-\,\d_tv(\tau,x)\,-\,\sum_{j=1}^nA^*_j(\tau)\,\d_jv(\tau,x)\,=\,g(\tau,x)\qquad\mbox{ on }\quad [0,t]\times\R^n\,,
$$ 
with datum $v(t)=0$. Remark that the existence and uniqueness of $v$ is guaranteed by Theorem \ref{th:speed} (keep in mind also Remark \ref{r:adj} and
Lemma \ref{l:symm}), which in addition implies ${\rm supp}\,v\,\subset\,B_0$.

Applying $u^k$ to the previous expression and integrating in time, we find
$$
\int^t_0<g\,,\,u^k>\,d\tau\,=\,\int^t_0<v\,,\,f^k>\,d\tau\,+\,<v(0)\,,\,u_0^k>\quad\longrightarrow\;0
$$
for $k\ra+\infty$, due to our hypotheses. This convergence holds true for all $g$ belonging to the space $L^1\bigl([0,t];\mc{A}'(B_t)\bigr)$, and it is uniform
on bounded subsets  $\mc{B}$ of $L^1\bigl([0,t];\mc{A}'(B_t)\bigr)$. Indeed,
this immediately follows from Theorem \ref{th:speed}: if $g\in\mc B$, then the family of corresponding solutions $v_g$ to the adjoint system is bounded
in the space $\mc{C}\bigl([0,t];\mc{A}'\bigr)$, and this is still enough to recover the previous convergence for $k\ra+\infty$.

In the end, we deduce that $u^k\longrightarrow0$ in $L^\infty\bigl([0,t];\mc A(B_t)\bigr)$, and then also in $\mc{C}\bigl([0,t];\mc{A}(B_t)\bigr)$
(recalling that $u^k$ solves $Lu^k=f^k$).

The lemma is now proved.
\end{proof}

Thanks to the previous lemma, we can derive a result on the \emph{domain of dependence}, which yields to \emph{local uniqueness} of the solution
to problem \eqref{eq:Cauchy}.
\begin{thm} \label{th:local-u}
 Let $L$ be the operator defined by formula \eqref{def:Lu}, with coefficients $\bigl(A_j\bigr)_j$ 
verifying property \eqref{hyp:coeff-L1}. Suppose that $L$ is microlocally symmetrizable, in the sense of Definition \ref{d:micro_symm},
and let $S$ be its symmetrizer. \\
Suppose also that either condition $(i)$ or condition $(ii)$ of Theorem \ref{th:speed} is fulfilled. \\
Let $u_0\in\mc{A}'$ and $f\in L^1\bigl([0,T];\mc{A}'\bigr)$, and let
$u\in\mc{C}\bigl([0,T];\mc A'\bigr)$ be the corresponding solution to problem \eqref{eq:Cauchy} (given by Theorem \ref{th:speed}).

Assume that $u_0(x)\equiv0$ and $f(t,x)\equiv0$ for $|x-x_0|<r_0$, for some $r_0>0$ and point $x_0\in\R^n$ fixed. Then
$$
u(t,x)\equiv0\qquad\qquad\mbox{ for all }\qquad |x-x_0|\,<\,r_0\,-\,2\,\sqrt{\Lambda}\int^t_0\alpha(\tau)\,d\tau\,,
$$
for all $t\in\,]0,T]$ such that $\int^t_0\alpha(\tau)\,d\tau\,<\,r_0/\bigl(2\,\sqrt{\Lambda}\bigr)$. 
\end{thm}

\begin{proof}
Let $\vrho(t)$ be defined as in Lemma \ref{l:dependence}, and let $t\in\,]0,T]$ such that $\vrho(t)>0$. We want to prove that
$u(t)=0$ in $\mc{A}'\bigl(B_t(x_0)\bigr)$, where we have denoted $B_t(x_0)=B\bigl(x_0,\vrho(t)\bigr)$ the ball of center $x_0$
and radius $\vrho(t)$.

This is equivalent to show the convergence
\begin{equation} \label{eq:to-prove}
<u(t)\,,\,w^k>\;\longrightarrow\;0
\end{equation}
for all sequence $\bigl(w^k\bigr)_k\,\subset\,\mc A$ such that (for some $\veps>0$)
$$
w^k\ra0\qquad\qquad \mbox{ in }\quad \mc A\left(\R^n\setminus\oline{B(x_0,\vrho(t)-\veps)}\right)\,.
$$

Then, fix a sequence $\bigl(w^k\bigr)_k$ as above, and let $v^k$ be the corresponding solution to the adjoint problem
$L^*v^k=0$ on $[0,t]\times\R^n$, with $L^*$ defined in \eqref{def:L*}, and initial datum $v^k(t)=w^k$. 

Thanks to Lemma \ref{l:dependence}, we deduce that $v^k\ra0$ in $\mc{C}\bigl([0,t];\mc A(U)\bigr)$, where, for notational convenience,
we have set $U\,:=\,\R^n\setminus\oline{B(x_0,r_0-\veps)}$.

On the other hand, let us apply $u$ to the system $L^*v^k=0$, and integrate over $[0,t]$: taking into account the initial conditions, direct
computations lead us to
$$
-\,<u(t)\,,\,w^k>\,+\,<u_0\,,\,v^k(0)>\,+\,\int^t_0<f\,,\,v^k>\,d\tau\,=\,0\,.
$$
Property \eqref{eq:to-prove} follows then from the hypotheses on $u_0$ and $f$.
\end{proof}

To conclude, we show a \emph{local existence and uniqueness} result in the space of analytic functions.
\begin{thm} \label{th:local-eu}
 Let $L$ be the operator defined by formula \eqref{def:Lu}, with coefficients $\bigl(A_j\bigr)_j$ 
verifying property \eqref{hyp:coeff-L1}. Suppose that $L$ is microlocally symmetrizable, in the sense of Definition \ref{d:micro_symm},
and let $S$ be its symmetrizer. \\
Suppose also that either condition $(i)$ or condition $(ii)$ of Theorem \ref{th:speed} is fulfilled. \\
Fix a point $x_0\in\R^n$ and a real number $r_0>0$, and denote $B_0=B(x_0,r_0)$. Defined the function $\vrho$ as in \eqref{def:rho}, set also $B_t=B\bigl(x_0,\vrho(t)\bigr)$. \\
Finally, assume that $u_0\in\mc{A}(B_0)$ and $f\in L^1\bigl([0,T];\mc{A}(B_0)\bigr)$.

Then, there exists a unique solution $u$ to problem \eqref{eq:Cauchy} on the conoid
$$
\Gamma\,:=\,\Bigl\{(t,x)\,\in\,[0,T]\times\R^n\;\bigl|\;\vrho(t)\,>\,0\;,\;|x-x_0|\,<\,\vrho(t)\Bigr\}\,,
$$
which belongs to the space $\mc{C}\bigl([0,t];\mc{A}(B_t)\bigr)$ for any time $t\in[0,T]$ such that $\vrho(t)>0$.
\end{thm}

\begin{proof}
We use the density of $\mc{A}$ in $\mc{A}(B_0)$ (see Appendix E of \cite{C-DG-S}). So, there exist sequences $\bigl(u_{0}^k\bigr)_k\subset\mc{A}$
and $\bigl(f^k\bigr)_k\subset L^1\bigl([0,T];\mc{A}\bigr)$ such that
$$
u_0^k\,\longrightarrow\,u_0\quad\mbox{ in }\;\mc{A}(B_0)\qquad\qquad\mbox{ and }\qquad\qquad
f^k\,\longrightarrow\,f\quad\mbox{ in }\;L^1\bigl([0,T];\mc{A}(B_0)\bigr)\,.
$$

For any $k\in\N$, let $u^k\,\in\,\mc{C}\bigl([0,T];\mc{A}\bigr)$ be the solution to problem \eqref{eq:Cauchy}, related to the data $u_0^k$ and $f^k$.
Recall that, fixed $k$, the existence of such a $u^k$ is guaranteed by Theorem \ref{th:A} here above.

On the other hand, if we apply Lemma \ref{l:dependence} to the families $\bigl(u_0^k-u_0^h\bigl)_{h,k}$ and $\bigl(f^k-f^h\bigl)_{h,k}$,
we immediately discover that the sequence of $\bigl(u^k\bigr)_k$ is a Cauchy sequence in $\mc{C}\bigl([0,T];\mc{A}(B_t)\bigr)$,
and hence it converges to some $u$ in this space. By linearity of the system, it is easy to check that $u$ is indeed a solution
of our problem \eqref{eq:Cauchy} on the conoid $\Gamma$.

\medbreak
The uniqueness of such a solution is a straightforward consequence of Theorem \ref{th:local-u}.

As a matter of fact, let $u_1$ and $u_2$ are two solutions related to the same data $u_0$ and $f$, verifying the hypotheses of Theorem \ref{th:local-eu}.
Then, their difference $\delta u\,:=\,u_1-u_2$ belongs to $\mc{A}(B_t)$ (for any $t$ such that $\vrho(t)>0$), and hence also to $\mc{A}'$.

Moreover, by linearity of our system, $\delta u$ satisfies system \eqref{eq:Cauchy} with initial datum and external force being both identically $0$.
Then, from Theorem \ref{th:local-u} one gets $\delta u\,\equiv\,0$.
\end{proof}

{\small

 }


\begin{thebibliography}{xxx}

\bibitem{B-C-D} H. Bahouri, J.-Y. Chemin, R. Danchin: 
\textit{``Fourier Analysis and Nonlinear Partial Differential Equations''}.
Grundlehren der Mathematischen Wissenschaften (Fundamental Principles of Mathematical Sciences),
{\bf 343}, Springer, Heidelberg (2011).

\bibitem{C-DG-S} F. Colombini, E. De Giorgi, S. Spagnolo:
{\it Sur les \'equations hyperboliques avec des coefficients qui ne d\'ependent que du temps}.
Ann. Scuola Normale Sup. Pisa Cl. Scienze (4), {\bf 6} (1979), n. 3, 511-559.

\bibitem{C-DS-F-M_tl} F. Colombini, D. Del Santo, F. Fanelli, G. M\'etivier:
{\it Time-dependent loss of derivatives for hyperbolic operators with non-regular coefficients}.
Comm. Partial Differential Equations, {\bf 38} (2013), n. 10, 1791-1817.

\bibitem{C-DS-F-M_wp} F. Colombini, D. Del Santo, F. Fanelli, G. M\'etivier:
{\it A well-posedness result for hyperbolic operators with Zygmund coefficients}.
J. Math. Pures Appl. (9), {\bf 100} (2013), n. 4, 455-475. 

\bibitem{C-DS-F-M_Z-sys} F. Colombini, D. Del Santo, F. Fanelli, G. M\'etivier:
{\it The well-posedness issue in Sobolev spaces for hyperbolic systems with Zygmund-type coefficients}.
Comm. Partial Differential Equations, {\bf 41} (2015), n. 11, 2082-2121.

\bibitem{C-DS-F-M_LL-sys} F. Colombini, D. Del Santo, F. Fanelli, G. M\'etivier:
{\it On the Cauchy problem for microlocally symmetrizable hyperbolic systems with log-Lipschitz coefficients}.
In preparation (2015).

\bibitem{C-L} F. Colombini, N. Lerner:
{\it Hyperbolic operators with non-Lipschitz coefficients}.
Duke Math. J., {\bf 77} (1995), 657-698.

\bibitem{C-M} F. Colombini, G. M\'etivier:
{\it The Cauchy problem for wave equations with non-Lipschitz coefficients; application to continuation of solutions of
some nonlinear wave equations}.
Ann. Sci. \'Ecole Norm. Sup. (4) {\bf 41} (2008), n. 2, 177-220.

\bibitem{C-M_2015} F. Colombini, G. M\'etivier:
{\it Couterexamples to the well-posedness of the Cauchy problem for hyperbolic systems}.
Anal. PDE, {\bf 8} (2015), n. 2, 499-511.


\bibitem{F-Z} F. Fanelli, E. Zuazua:
{\it Weak observability estimates for $1$-D wave equations with rough coefficients},
Ann. Inst. H. Poincar\'e Anal. Non Lin\'eaire, {\bf 32} (2015), n. 2, 245-277.

\bibitem{Fried_1954} K. O. Friedrichs:
{\it Symmetric hyperbolic linear differential equations}.
Comm. Pure Appl. Math., {\bf 7} (1954), 345-392.

\bibitem{Fried_1958} K. O. Friedrichs:
{\it Symmetric positive linear differential equations}.
Comm. Pure Appl. Math., {\bf 11} (1958), 333-418.

\bibitem{Iv-Pet} V. J. Ivri\u{\i}, V. M. Petkov:
{\it Necessary conditions for the correctness of the Cauchy problem for non-strictly hyperbolic equations}.
Uspehi Math. Nauk, {\bf 29} (1974), n. 5(179), 3-70. English translation: Russian Math. Surveys, {\bf 29} (1974), n. 5, 1-70. 

\bibitem{J-M-R_2005} J.-L. Joly, G. M\'etivier, J. Rauch:
{\it Hyperbolic domains of determinacy and Hamilton-Jacobi equations}.
J. Hyperbolic Differ. Equ., {\bf 2} (2005), n. 3, 713-744.

\bibitem{M-2008} G. M\'etivier:
\textit{``Para-differential calculus and applications to the Cauchy problem for nonlinear systems''}.
Centro di Ricerca Matematica ``Ennio De Giorgi'' (CRM) Series, {\bf 5}, Edizioni della Normale, Pisa (2008).

\bibitem{M_2014} G. M\'etivier:
{\it $L^2$ well-posed Cauchy problems and symmetrizability of first order systems}.
J. \'Ec. Polytech. Math., {\bf 1} (2014), 39-70.

\bibitem{Rauch_2005} J. Rauch:
{\it Precise finite speed with bare hands}.
Methods Appl. Anal., {\bf 12} (2005), n. 3, 267-277. 

\bibitem{Rauch-book} J. Rauch:
\textit{``Hyperbolic Partial Differential Equations and Geometric Optics''}.
Graduate Studies in Mathematics, {\bf 133}, American Mathematical Society, Providence, RI (2012).

\bibitem{Tar} S. Tarama:
{\it Energy estimate for wave equations with coefficients in some Besov type class}.
Electron. J. Differential Equations (2007), Paper No. 85, 12 pp. (electronic).

\bibitem{Treves} F. Treves:
\textit{``Basic linear partial differential equations''}.
Pure and Applied Mathematics, Vol. \textbf{62}, Academic Press, New York-London (1975).

\bibitem{Waters} A. Waters:
{\it Multi-dimensional low regularity observability estimates}.
Submitted (2015), arXiv preprint \texttt{http://arxiv.org/abs/1512.02180}.

\end{thebibliography}
\end{document}